\documentclass[11pt]{article}

\usepackage[T1]{fontenc}
\usepackage[utf8]{inputenc}
\usepackage[english]{babel}

\usepackage[dvipsnames,svgnames]{xcolor}

\usepackage{amsmath,amssymb,amsfonts,amsthm,thmtools,upgreek,verbatim}
\usepackage{stmaryrd}
\usepackage{mathtools}
\usepackage{scalerel}
\usepackage{mleftright}
\mleftright
\let\epsilon\varepsilon
\let\phi\varphi

\usepackage{bbm}
\usepackage{bm}

\DeclareMathAlphabet\mbc{OMS}{cmsy}{b}{n}
\DeclareMathAlphabet\mathbfcal{OMS}{cmsy}{b}{n}

\usepackage[overload]{empheq}

\usepackage{booktabs}

\usepackage{enumerate,enumitem}

\usepackage{subcaption}
\usepackage[font={small}]{caption}

\usepackage{algorithm,algpseudocode}

\usepackage{fancyhdr}
\usepackage[a4paper]{geometry}
\usepackage{url}
\usepackage{hyperref}
\hypersetup{
    colorlinks = true,
    linkcolor = DarkOrchid,
    citecolor = BrickRed,
    urlcolor = blue
}
\usepackage[capitalise,noabbrev,nameinlink]{cleveref}
\newcommand{\crefpart}[2]{%
  \hyperref[#2]{\namecref{#1}~\labelcref*{#1}~\ref*{#2}}%
}
\crefname{equation}{}{}
\crefname{enumi}{}{}

\theoremstyle{definition}
\newtheorem{theorem}{Theorem}[section]
\newtheorem{definition}[theorem]{Definition}
\newtheorem{proposition}[theorem]{Proposition}
\newtheorem{corollary}[theorem]{Corollary}

\newtheorem{lemma}[theorem]{Lemma}
\newtheorem{remark}[theorem]{Remark}
\newtheorem{example}[theorem]{Example}

\usepackage{tcolorbox}

\usepackage{tikz}
\usepackage{tikz-cd}
\usetikzlibrary{decorations.markings,arrows}
\usetikzlibrary{shapes}
\usepackage{adjustbox}
\usepackage{pgf}
\usepackage{pgfplots}
\pgfplotsset{compat=1.18}

\DeclareMathOperator{\Fix}{Fix}
\DeclareMathOperator{\ran}{ran}
\DeclareMathOperator{\spa}{span}

\DeclareMathOperator{\Id}{Id}
\DeclareMathOperator{\Prox}{prox}
\DeclareMathOperator{\Diag}{diag}
\DeclareMathOperator{\Lap}{Lap}

\DeclareMathOperator*{\argmin}{argmin}

\newcommand{\Hi}{\mathcal{H}}
\newcommand{\E}{\mathbb{E}}

\newcommand{\R}{\mathbb{R}}
\newcommand{\C}{\mathbb{C}}
\newcommand{\N}{\mathbb{N}}
\newcommand{\Ei}{\mathcal{E}}

\newcommand{\norm}[1]{\left\lVert #1\right\rVert}
\newcommand{\oDelta}{B}
\newcommand{\oxi}{\mathcal{C}}
\newcommand{\ochi}{\mathcal{D}}
\newcommand{\ot}[1]{#1_{\scaleto{\otimes}{4pt}}}
\newcommand{\okernel}{\Lambda}

\usepackage[doi=true,isbn=false,url=false,eprint=false,sortcites=true,maxnames=5,giveninits=true,sorting=nyt,backend=biber]{biblatex}
\usepackage{csquotes}
\DeclareNameAlias{author}{family-given}
\DeclareFieldFormat{doi}{\url{https://doi.org/#1}}
\renewbibmacro{in:}{}
\DeclareFieldFormat[article]{title}{#1}

\DeclareFieldFormat[article]{volume}{\textbf{#1}}
\AtEveryBibitem{%
  \ifentrytype{article}{%
    \clearfield{number}%
    \clearfield{issue}%
  }{}%
}
\DeclareFieldFormat[article]{pages}{#1}
\renewbibmacro*{journal+issuetitle}{%
  \usebibmacro{journal}%
  \setunit{\addspace}%
  \printfield{volume}%
  \setunit{\addcomma\space}%
  \printfield{pages}%
  \setunit{\addspace}%
  \printtext[parens]{\usebibmacro{date}}%
  \newunit
}
\renewbibmacro*{note+pages}{}

\newcommand*\samethanks[1][\value{footnote}]{\footnotemark[#1]}
\title{On the optimal relaxation parameter of \\ graph-based splitting methods for subspaces}
\author{
    Francisco J. Arag\'on-Artacho\thanks{University of Alicante, Department of Mathematics, Alicante, \textsc{Spain}. e-mail:~\href{mailto:francisco.aragon@ua.es}{francisco.aragon@ua.es},\\\href{mailto:cesar.lopez@ua.es}{cesar.lopez@ua.es}}
    \and C\'esar L\'opez-Pastor\samethanks[1]
}
\date{}
\hypersetup{linkcolor = DarkOrange}

\begin{document}

\maketitle

\begin{abstract}
   In this paper, we investigate the behavior of the family of graph-based splitting algorithms specialized to the problem of finding a point in the intersection of linear subspaces. The algorithms in this family, which encompasses several classical methods such as the Douglas--Rachford algorithm, are defined by a connected graph and a subgraph. Our main result establishes that when the graph and subgraph coincide, the optimal relaxation parameter is exactly $1$, thereby extending known results for the Douglas--Rachford algorithm to a much broader class of methods. Our analysis hinges on some properties of iso-averaged linear operators, which are defined as the average of an isometry and the identity, and are characterized by a specific symmetry of the norm of their relaxation.
\end{abstract}



\section{Introduction}

Let $\E$ be a Euclidean space, and let $f_1,\ldots,f_n:\E\to(-\infty,\infty]$ be proper, convex and lower semicontinuous functions. We are interested in the following optimization problem:
\begin{equation}\label{problem}
  \operatorname*{Min }_{x\in\E} f_1(x)+\cdots+f_n(x).
\end{equation}
By Fermat's rule, this is equivalent to finding some $x\in\E$ such that $0\in\partial(f_1+\cdots+f_n)(x)$, where $\partial f$ denotes the subdifferential. Under some constraint qualification (e.g., \cite[Theorem~23.8]{rockafellar1970convex}), we obtain an equivalent reformulation of \cref{problem} by rewriting it as
\begin{equation}\label{problem2}
  \text{Find }x\in\E\text{ such that }0\in \partial f_1(x)+\cdots+\partial f_n(x).
\end{equation}

In recent years, numerous algorithms have been designed to solve \cref{problem2} by treating each term of the sum separately. These methods, commonly referred to as \emph{splitting algorithms}, can be
applied in more general settings where $\E$ is a Hilbert space and the subdifferentials are replaced by arbitrary maximally monotone operators (see, e.g.,
\cite{tam2023frugal,condat2023proximal,campoy,malitsky2023resolvent,ryu20}).

For $n=2$, Lions and Mercier \cite{LM79} were the first to develop the celebrated Douglas--Rachford splitting algorithm (DRA). Given a starting point $v^0\in\E$ and a \emph{relaxation parameter} $\theta\in(0,2)$, the algorithm generates the following sequences of points:
\begin{align*}[left=\empheqlbrace]
x_1^{k+1}&=\Prox_{f_1}(v^k),\\
x_2^{k+1}&=\Prox_{f_2}(2x_1^{k+1}-v^k),\\
v^{k+1}&=v^k+\theta(x_2^{k+1}-x_1^{k+1}),
\end{align*}
for $k=0,1,\ldots$. Here ``$\Prox$'' refers to the \emph{proximal map}, which is defined for a function $f$ as $\Prox_f=(\Id+\partial f)^{-1}=\argmin_{z\in\E}\bigl\{f(z)+\frac{1}{2}\norm{z-\cdot}^2\bigr\}$, where $\Id$ denotes the identity map in~$\E$.
The DRA provides three convergent sequences with the following properties: $(v^k)_{k\in\N}$ converges to some point $v^\star\in\E$, while $(x^k_1)_{k\in\N}$ and $(x^k_2)_{k\in\N}$ both converge to the same limit point $x^\star:=\Prox_{f_1}(v^\star)=\Prox_{f_2}(2x^\star-v^\star)$, which is a solution to \cref{problem2}.

In this work, we focus on the particular case in which $f_i=\imath_{U_i}$ is the \emph{indicator function} of a linear subspace $U_i\subseteq\E$. In this setting, we have that $\Prox_{f_i}=P_{U_i}$ is the \emph{projection operator} onto~$U_i$. If we denote by $R_{U_i}:=2P_{U_i}-\Id$ the \emph{reflector} to $U_i$, then the DRA can be rewritten as
\begin{equation*}
  v^{k+1}=T_\theta(v^k):=\frac{\theta}2R_{U_2}R_{U_1}(v^k)+\left(1-\frac{\theta}2\right)v^k,
\end{equation*}
which can be understood as a fixed-point iteration of the linear map $T_\theta$. Furthermore, it was proved in~\cite{cosineDR,relaxDR} that, for all $\theta\in(0,2)$, $T_\theta^k$ converges linearly to $P_{\Fix T}$ with rate $\gamma\in(0,1)$; that is, there exists some $M>0$ such that $\norm{T_\theta^k-P_{\Fix T}}\leq M\gamma^k$ eventually, where $\norm{\cdot}$ denotes the operator norm. The infimum of all such rates is called the \emph{optimal rate of linear convergence} of $T_\theta$ and is denoted by $\gamma(T_\theta)$.

It is thus natural to question how the relaxation parameter $\theta$ affects the performance of the algorithm. For $\theta=1$, it was shown in \cite[Theorem 4.1]{cosineDR} that the optimal rate of convergence is $\gamma(T_1)=c_F(U_1,U_2)$, the \emph{cosine of the Friedrichs angle} between $U_1$ and $U_2$
(see \cite[Definition~9.4]{deutsch}). More generally, by \cite[Theorem 3.10]{relaxDR}, the optimal rate of convergence of the DRA for $\theta\in(0,2)$ is
\begin{equation}\label{optimal theta}
  \gamma(T_\theta)=\sqrt{\theta(2-\theta)c_F(U_1,U_2)^2+(1-\theta)^2}.
\end{equation}
We observe in \cref{optimal theta} that $\gamma(T_\theta)$ can be treated as a function of $\theta$, which attains its minimum at $\theta=1$, from where we deduce that this is the optimal relaxation parameter of the DRA.

The DRA can be applied to $n$ functions, for any arbitrary $n \in \mathbb{N}$, using Pierra's reformulation~\cite{pierra}, which yields an equivalent problem involving two operators in the product space. However, this produces the noticeable drawback of increasing the number of variables $v^k$ that define the fixed-point map $T$ (which is referred to as \emph{lifting}). Ryu's seminal work \cite{ryu20} proved that for $n=3$, it is impossible to define a splitting algorithm without increasing the number of variables $v^k$. At the same time, he provided the first splitting algorithm with a lifting of two variables $v_1^k$ and $v_2^k$.
Later, Campoy \cite{campoy} and Malitsky and Tam~\cite{malitsky2023resolvent} independently designed the first splitting methods for an arbitrary number $n$ of operators with lifting $n-1$, with the latter authors proving that the minimal lifting for this class of algorithms is indeed $n-1$.

Recognizing the structural similarities among recently developed splitting algorithms, Bredies, Chenchene and Naldi~\cite{graph-drs} defined a unifying theoretical framework making use of the degenerate preconditioned tools from \cite{degenerate-ppp}. Their approach models the relationship between the variables $x_i^k$ and $v_j^k$ using a pair of directed graphs.
Many new algorithms can be devised using this framework, such as the \emph{parallel down algorithm} or the \emph{sequential algorithm} \cite{graph-drs}. For $n=3$ operators, the \emph{complete algorithm} was also proposed in that work. Later, it was extended to an arbitrary number of operators in \cite{graph-fb}, allowing forward computations of cocoercive functions.

A natural question thus emerges: which choice of graphs provides the best performance? Based on the exact expression for the limit point of a general graph-based splitting algorithm from~\cite{graph-linear}, numerical experiments were carried out in~\cite{graphnumerical} to study the optimal relaxation parameter and compare the performance. To this end, the number of iterations $k_\theta$ needed to satisfy the stopping criterion $\|T_\theta^{k_\theta}x-P_{\Fix T}x\|<\varepsilon$ for some $\varepsilon>0$ was computed. In particular, the authors observed in~\cite[Figure~2]{graphnumerical} that the equality $k_\theta=k_{2-\theta}$ holds whenever $G$ and $G'$ coincide, and that the optimal relaxation parameter is always $\theta=1$.

The goal of this paper is to provide the theoretical justification for these empirical observations and formally prove that the optimal relaxation parameter is indeed $\theta=1$ whenever both graphs coincide. The main contributions of our work are summarized as follows:
\begin{itemize}
  \item We introduce the concept of an iso-averaged linear map and compute the rate of linear convergence for any relaxation parameter. Their optimal relaxation parameter is $\theta=1$.
  \item We establish that the operator defining a graph-based splitting algorithm is iso-averaged if and only if the graphs are equal, thereby determining the optimal relaxation parameter.
  \item We present several pathological examples of graph-based algorithms whose fixed-point maps fail to be iso-averaged or even normal.
\end{itemize}

The remainder of this paper is organized as follows. In \cref{sec:iso}, we present some properties of normal operators and introduce the stronger notion of iso-averagedness. We then characterize iso-averaged maps with a certain symmetry of the norm and compute the linear rate of convergence of fixed-point algorithms defined by this class of maps. We specialize in \cref{sec:graph} to the case of graph-based splitting algorithms and state the main result proving the equivalence between iso-averagedness and the equality of the two graphs, deducing the optimal relaxation parameter.

\section{Relaxation of normal and iso-averaged maps}\label{sec:iso}

Given a linear map $T:\E\to\E$ and $\theta\in\mathbb{R}$, we define the $\theta$-\emph{relaxed map} as $T_\theta:=\theta T+(1-\theta)\Id$, where $\theta$ is known as the \emph{relaxation parameter}. When $\theta\neq0$, it is straightforward to see that $\Fix T_\theta=\Fix T$. Not only that, but the normality of  $T_\theta$ is also inherited by $T$, as the following lemma states. We recall that $T$ is said to be \emph{normal} if $T^*T=TT^*$, where $T^*$ denotes the adjoint mapping.

\begin{lemma}\label{Ttheta normal}
 A linear map $T:\E\to\E$  is normal if and only if $T_\theta$ is normal for all $\theta\in\mathbb{R}$.
\end{lemma}

\begin{proof}
By the definition of $T_\theta$, we have that
 \begin{align*}
 	T_\theta^*T_\theta&
 	=\theta^2T^*T+\theta(1-\theta)(T^*+T)+(1-\theta)^2\Id,\\
 	T_\theta T_\theta^*&=\theta^2TT^*+\theta(1-\theta)(T^*+T)+(1-\theta)^2\Id,
 \end{align*}
so $T_\theta^*T_\theta-T_\theta T^*_\theta=\theta^2\left(T^*T-TT^*\right)$, and the result follows.
\end{proof}

Normal mappings on finite-dimensional complex inner product spaces can be decomposed using their eigenvalues via the spectral theorem. To this end, let
 $\E_\C$ and $T_\C:\E_\C\to \E_\C$ be the \emph{complexifications} of $\E$ and $T$ defined by
\begin{equation}\label{complexif}
T_\C z:=T u+iT v,\quad \text{for all }z=u+iv\in \E_\C,
\end{equation}
see, e.g.,~\cite[Exercises~1B.8 and 3B.33]{Axler}. We denote by $\sigma(T)$ the set of complex eigenvalues of the linear map $T_\C$. Then, the set of eigenvalues of $(T_\theta)_\C$ is given by
\begin{equation}\label{lambda theta}
\sigma(T_\theta)=\{\lambda_\theta:=\theta\lambda+1-\theta\mid\lambda\in\sigma(T)\}.
\end{equation}
Using the spectral decomposition of $T_\C$, we can prove the following result.

\begin{proposition}\label{convex}
  If $T$ is normal, then $\theta\mapsto\norm{T^k_\theta x}$ is convex for all $x\in\E$ and $k\in\mathbb{N}$.
\end{proposition}

\begin{proof}
By the spectral theorem  (e.g.,~\cite[Theorem~9 in \S9.5]{Hoffman}), it holds that
\[T_\C=\sum_{i=1}^t\lambda P_{\ker(T_\C-\lambda_i \Id)},\]
where we index $\sigma(T)=\{\lambda_i\}_{i=1}^t$.
Let $(T_\C)_\theta^k:=(T_\C)_\theta\circ\cdots\circ (T_\C)_\theta$ be the composition of $k$ identical $\theta$-relaxations of the map~$T_\C$.
Now, by~\cite[Theorem~11 in \S9.5]{Hoffman}, we deduce
\begin{equation}\label{spectral T theta}
  (T_\C)_\theta^k=\sum_{i=1}^t(\theta\lambda_i+1-\theta)^kP_{\ker(T_\C-\lambda_i \Id)}.
\end{equation}
Further, for any $x\in\mathbb{E}$, we can express $x=\sum_{i=1}^tx_{\lambda_i}$ with $x_{\lambda_i}\in \ker(T_\C-\lambda_i \Id)$. Then, by \cref{complexif,spectral T theta},
  \[\norm{T_\theta^kx}=\norm{(T_\C)_\theta^kx}=\norm{\sum_{i=1}^t(\theta\lambda_i+1-\theta)^kx_{\lambda_i}}=\sqrt{\sum_{i=1}^t|\theta\lambda_i+1-\theta|^{2k}\norm{x_{\lambda_i}}^2}=\norm{g(\theta)}^k_{2k},\]
  where $g(\theta):=\big(|\theta\lambda_i+1-\theta|\norm{x_{\lambda_i}}^{1/k}\big)_{i=1}^t$ and $\|\cdot\|_{2k}$ denotes the $2k$-norm. By~\cite[Example~3.14]{Boyd2004}, the function $\norm{g(\theta)}_{2k}$ is convex, and so is $\norm{T_\theta^kx}=\norm{g(\theta)}^k_{2k}$.
\end{proof}

\begin{example}[Necessity of normality in \cref{convex}]\label{not normal}
  Let $\E$ be $\R^2$ equipped with the standard inner product and consider
	\[T=\begin{bmatrix}
		0 & 1 \\ 0 & 0
\end{bmatrix}\quad \text{and}\quad  x=(0,1).\]
Then, we compute $T_\theta^2x=\theta^2T^2x+2\theta(1-\theta)Tx+(1-\theta)^2x=\big(2\theta(1-\theta),(1-\theta)^2\big)$. Hence, defining $f(\theta):=\norm{T^2_\theta x}=\sqrt{4\theta^2(1-\theta)^2+(1-\theta)^4}$, one clearly sees that $f$ is not convex (see \cref{necessity normality}). Indeed, $f(1/2)=\sqrt{5}/4>1/2=(f(0)+f(1))/2$.

\begin{figure}[h]
	\centering
\includegraphics[scale=0.95]{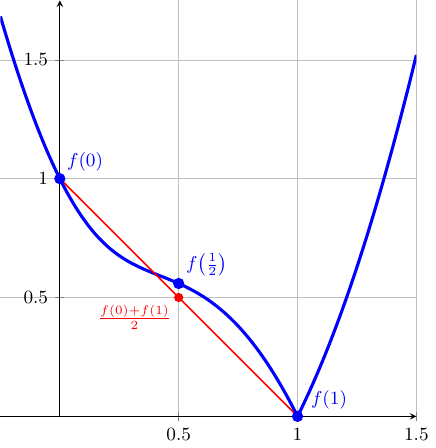}
\caption{Graph of the nonconvex function $f(\theta)=\norm{T^2_\theta x}$ (blue) given in \cref{not normal}}
\label{necessity normality}
\end{figure}
\end{example}

In addition to providing favorable properties for relaxed maps, normality facilitates the study of the linear convergence rate of the sequence $(T^k)_{k\in\N}$. Indeed, suppose that $T^k$ converges to~$T_\star$, i.e., $\norm{T^k-T_\star}\to0$. If $T$ is normal, by \cite[Corollary~2.7(ii),Theorem~2.17]{relaxDR}, $T^k$ converges linearly to $T_\star=P_{\Fix T}$ with optimal rate
\begin{equation*}
  \rho_1(T):=\max\left\{|\lambda|\mid\lambda\in \{0\}\cup\sigma(T)\setminus\{1\}\right\}\in [0,1),
\end{equation*}
that is, there is some $M>0$ such that $\norm{T^k-T_\star}\leq M\rho_1(T)^k$. The constant $\rho_1(T)$ is called the \emph{subdominant radius} of $T$.  The same arguments can be applied to $T_\theta$, by virtue of \cref{Ttheta normal}, and $T^k_\theta\to P_{\Fix T}$ with optimal linear rate $\rho_1(T_\theta)$. In general, however, there is no simple way of computing the subdominant radius of $T_\theta$ in terms of $\rho_1(T)$. We will show that this is not the case for the following class of linear maps (see \cref{symmetry} below).

\begin{definition}
  A linear map $T:\E\to\E$ is \emph{iso-averaged} if $2T^*T=T+T^*$.
\end{definition}

It was shown in \cite[Proposition~3.5(iv)]{cosineDR} that the Douglas--Rachford splitting operator is iso-averaged, without giving a name to this property. The next result shows that every iso-averaged map $T:\E\to\E$ is normal and can be written as $T=(S+\Id)/2$, where $S:\E\to\E$ is an \emph{isometry} (i.e., $S^*S=\Id$), thus motivating the name. In addition, any iso-averaged map is \emph{firmly nonexpansive}, meaning that $2T-\Id$ is nonexpansive (see~\cite[Proposition~4.4]{bauschke}).

\begin{proposition}\label{prop:iso-normal}
  A linear map $T:\E\to\E$  is iso-averaged if and only if $2T-\Id$ is an isometry. Thus, every iso-averaged map is normal and firmly nonexpansive.
\end{proposition}

\begin{proof}
  Denoting by $S:=2T-\Id$, we obtain that $T=(S+\Id)/2$. Hence, $T$ is iso-averaged if and only if $\frac{1}{2}(S^*+\Id)(S+\Id)=\frac12(S^*+S+2\Id)$. Rearranging the terms of this equation yields the desired equality $S^*S=\Id$. To show that $T$ is normal, notice that this is equivalent to proving that $S$ is normal. The condition $S^*S=\Id$ implies that $S^*$ is a left-inverse of $S$. Since $\E$ is finite dimensional, this also implies that $S^*$ is a right-inverse (and thus, $S$ is orthogonal). In particular, $\Id=S^*S=SS^*$ and $S$ is normal. Finally, since $S$ is orthogonal, we have $\norm{S}=1$ and the firm nonexpansiveness of $T$ is due to \cite[Corollary~4.5]{bauschke}.
\end{proof}

In light of \cref{prop:iso-normal}, it is natural to ask whether iso-averaged maps are always normal in Hilbert spaces. The next example demonstrates that this property fails to hold in the infinite-dimensional setting.

\begin{example}[Iso-averaged maps might not be normal in Hilbert spaces]
   Let $\Hi=\ell^2(\mathbb{N})$ and consider the right-shifted map $R:\Hi\to\Hi$ defined as $R(x_1,x_2,\ldots):=(0,x_1,x_2,\ldots)$. Then $R$ is an isometry which is not normal, so $T:=(R+\Id)/2$ is iso-averaged but not normal. Indeed, one checks that $R^*=L$, the left-shifted map, defined as $L(x_1,x_2,\ldots):=(x_2,x_3,\ldots)$. Thus,
  \[R^*Rx=LR(x_1,x_2,\ldots)=L(0,x_1,x_2,\ldots)=(x_1,x_2,\ldots)=x,\]
  so $R^*R=\Id$. On the other hand,
  \[RR^*x=RL(x_1,x_2,\ldots)=R(x_2,x_3,\ldots)=(0,x_2,x_3\ldots)\neq x.\]
  Since $R^*R\neq RR^*$, it follows that $R$ is not normal, and consequently, neither is $T$.
\end{example}

The next example shows that the relaxation of an iso-averaged map need not have this property. At the same time, it provides a simple example of a normal map that is not iso-averaged.

\begin{example}[The relaxed iso-averaged map might not be iso-averaged]\label{relaxed iso-averaged}
Let $\E$ be $\R^2$ equipped with the standard inner product and consider the map
\[T=\begin{bmatrix}
1 & 0 \\ 0 & 0
\end{bmatrix}.\]
One can readily check that $T$ is iso-averaged. For any $\theta\in\mathbb{R}$, we have $T_\theta=\Diag(1,1-\theta)$, where $\Diag$ denotes the diagonal matrix. Observe that $T_\theta$ is normal for all $\theta\in\R$, but is iso-averaged if and only if $\theta\in\{0,1\}$.
\end{example}

If $T$ is iso-averaged, then for every eigenvalue $\lambda\in\sigma(T)$, it holds that
\begin{equation}
  2\overline{\lambda}\lambda=\lambda+\overline{\lambda}\text{, equivalently }|\lambda|^2=\operatorname{Re}(\lambda),\label{isoav lambda}
\end{equation}
where $\overline{\lambda}$ denotes the complex conjugate of $\lambda$ (recall that $\sigma(T^*)=\left\{\overline{\lambda}\mid \lambda\in\sigma(T)\right\}$).
This means that all eigenvalues $\lambda$ lie on the circle of radius $1/2$ centered at $1/2+0i$. Hence, $|\lambda|\leq1$ and $\lambda=1$ is the only eigenvalue on the unit circle. By \cite[Fact 2.4]{relaxDR}, we deduce that $T^k$ converges to some $T_\star$, and by the previous discussion, it converges linearly to $P_{\Fix T}$ with optimal rate $\rho_1(T)$.

\begin{theorem}\label{symmetry}
  Let $T:\E\to\E$ be a linear map. The following are equivalent:
  \begin{enumerate}
    \item\label{symmetry1} $T$ is iso-averaged.
    \item\label{symmetry2} The function $\theta\mapsto\norm{T_\theta x}$ is symmetric with respect to $\theta=1$ for all $x\in\E$.
    \item \label{function} The function $\theta\mapsto\norm{T^k_\theta x}$ is symmetric with respect to $\theta=1$ for all $x\in\E$ and $k\in \N$.
 \end{enumerate}
 Moreover, if one of these conditions holds, for any $\theta\in(0,2)$, the iterated map $T^k_\theta\to P_{\Fix T}$ with optimal linear rate $\rho_1(T_\theta)$ satisfying
\begin{equation}\label{rate T theta}
  \rho_1(T)<\rho_1(T_\theta)=\sqrt{\theta(2-\theta)\rho_1(T)^2+(1-\theta)^2}<1,\quad\text{whenever }\theta\neq 1\text{ and }T\neq\Id.
\end{equation}
 Furthermore, the sequence $(\norm{T^k_\theta x})_{k\in\mathbb{N}}$ is strictly decreasing for all $x\notin\Fix T$ if $\theta\in(0,2)\setminus\{1\}$, and for all $x\notin\Fix T\oplus \ker T$ if $\theta=1$.
\end{theorem}

\begin{proof}
  Let us first show the equivalence between \cref{symmetry1} and \cref{symmetry2}. Observe that \cref{symmetry2} amounts to $\norm{T_\theta x}=\norm{T_{2-\theta} x}$ for all $x\in\E$. This identity is equivalent to
  \[\langle T_\theta^*T_\theta x,x\rangle=\langle T_\theta x,T_\theta x\rangle=\norm{T_\theta  x}^2=\norm{T_{2-\theta}  x}^2=\langle T_{2-\theta}^*T_{2-\theta} x,x\rangle,\quad\forall x\in\E,\]
  which holds if and only if $T_\theta^*T_\theta =T_{2-\theta}^*T_{2-\theta} $.
  Computing both maps, we obtain
  \begin{align*}
    T_{2-\theta}^*T_{2-\theta}&=(2-\theta)^2T^*T+(2-\theta)(\theta-1)(T^*+T)+(\theta-1)^2\Id,\\
    T_\theta^*T_\theta&=\theta^2T^*T+\theta(1-\theta)(T^*+T)+(1-\theta)^2\Id.
  \end{align*}
  Subtracting them, we get
  \begin{equation*}
    T_{2-\theta}^*T_{2-\theta}-T_\theta^*T_\theta=\big((2-\theta)^2-\theta^2\big)T^*T+\big((2-\theta)(\theta-1)-\theta(1-\theta)\big)(T+T^*).
  \end{equation*}
  Computing $(2-\theta)^2-\theta^2=4(1-\theta)$ and $(2-\theta)(\theta-1)-\theta(1-\theta)=2(\theta-1)$, we simplify
    \begin{equation}\label{symmetry isoav}
    T_{2-\theta}^*T_{2-\theta}-T_\theta^*T_\theta=2(1-\theta)\big(2T^*T-T-T^*\big).
  \end{equation}
  To sum up, \cref{symmetry2} is equivalent to saying that \cref{symmetry isoav} is zero, which is the same as \cref{symmetry1}.

  To see the equivalence of \cref{symmetry2} and \cref{function}, it is enough to observe that \cref{symmetry2} implies normality of~$T$ (by \cref{prop:iso-normal}), which yields
\begin{equation*}
  (T_{2-\theta}^k)^*T_{2-\theta}^k=(T_{2-\theta}^*)^kT_{2-\theta}^k=(T_{2-\theta}^*T_{2-\theta})^k=(T_\theta^*T_\theta)^k=(T_\theta^k)^*T_\theta^k,
\end{equation*}
and proves that $T_\theta^k$ satisfies \cref{function}, by our previous argumentation. As~\cref{function} trivially implies~\cref{symmetry2}, this proves the statement.

Now, since $T_\theta$ is normal, to obtain the rate of linear convergence of the sequence $(T^k_\theta)_{k\in\N}$ we only need to compute $\rho_1(T_\theta)$. Combining \cref{lambda theta,isoav lambda}, we deduce that the modulus of $\lambda_\theta\in\sigma(T_\theta)$ satisfies
  \begin{align*}
    |\lambda_\theta|^2=|\theta\lambda+1-\theta|^2=\theta^2|\lambda|^2+\theta(1-\theta)(\lambda+\overline\lambda)+(1-\theta)^2=\left(\theta^2+2\theta(1-\theta)\right)|\lambda|^2+(1-\theta)^2.
  \end{align*}
  Therefore, for all $\theta\in (0,2)$, we obtain
  \begin{align}
    |\lambda_\theta|&=\sqrt{\theta(2-\theta)|\lambda|^2+(1-\theta)^2}=\left\|\left(\sqrt{\theta(2-\theta)}|\lambda|,1-\theta\right)\right\| \label{lambda theta function1}\\
    &=\sqrt{|\lambda|^2+\left(\sqrt{1-|\lambda|^2}(1-\theta)\right)^2}=\left\|\left(|\lambda|,\sqrt{1-|\lambda|^2}(1-\theta)\right)\right\|.\label{lambda theta function2}
  \end{align}
  Then, by~\cref{lambda theta function2}, the function $\theta\mapsto |\lambda_\theta|$ is convex (recall again \cite[Example~3.14]{Boyd2004}). In fact, it is strictly convex if $|\lambda|\notin\{0,1\}$ because it can be expressed as the composition of the Euclidean norm with the affine function $\theta\mapsto\left(|\lambda|,\sqrt{1-|\lambda|^2}(1-\theta)\right)$, whose range does not lie on a ray through the origin (so strict convexity follows directly from the strict triangle inequality). Further, it is symmetric with respect to the vertical axis $\theta=1$, and has a minimum at $\theta=1$ with value $|\lambda_1|=|\lambda|$. The graphs of $|\lambda_\theta|$ for some chosen $|\lambda|$ are illustrated in \cref{rho1}.

\begin{figure}[ht]
\centering
\includegraphics{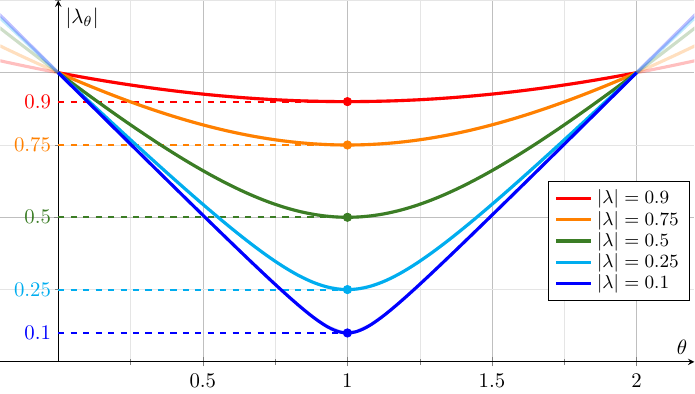}
\caption{Graphs of $|\lambda_\theta|$ for $|\lambda|\in\{0.1,0.25,0.5,0.75,0.9\}$}
\label{rho1}
\end{figure}
To prove \cref{rate T theta}, observe first that $\sigma(T)\setminus\{1\}\neq\varnothing$ if $T\neq\Id$. Further, if  $\theta\neq 1$, by~\cref{lambda theta function1} it holds that $|\lambda_\theta|>0$. Now, by~\cref{lambda theta function1}, the function $|\lambda|\mapsto |\lambda_\theta|$ is increasing for any fixed value of $\theta\in(0,2)$. Putting all this together and using~\cref{lambda theta}, we deduce that
\[
\rho_1(T_\theta)=\max_{\lambda\in \{0\}\cup\sigma(T_\theta)\setminus\{1\}}|\lambda|=\max_{\lambda\in\sigma(T)\setminus\{1\}} |\lambda_\theta|=\sqrt{\theta(2-\theta)\rho_1(T)^2+(1-\theta)^2}>\rho_1(T).
\]

Lastly, let us check that $\|T^{k+1}_\theta x\|<\|T^k_\theta x\|$. Recall that by \cref{prop:iso-normal}, $T=(S+\Id)/2$, where $S$ is an isometry. This implies that $T_\theta=S_{\theta/2}$ and $\norm{Sz}=\norm{z}$ for all $z\in\E$. Therefore, by \cite[Corollary~2.15]{bauschke}, we write
\begin{equation*}
    \norm{T^{k+1}_\theta x}^2=\norm{S_{\theta/2}(T^k_\theta x)}^2=\norm{T^k_\theta x}^2-\frac{\theta}2\bigg(1-\frac{\theta}2\bigg)\norm{(S-\Id)T^k_\theta x}^2.
\end{equation*}
Thus, the sequence is strictly decreasing if and only if $-\frac{\theta}2(1-\frac{\theta}{2})\norm{(S-\Id)T^k_\theta x}^2<0$. This requires $\theta\in(0,2)$. The condition $\norm{(S-\Id)T^k_\theta x}^2=0$ is equivalent to saying that $$T^k_\theta x\in \ker (S-\Id)=\Fix S=\Fix T.$$
If we decompose $x=x_F+x^\perp\in\Fix T\oplus(\Fix T)^\perp$, then $T^k_\theta x=x_F+T^k_\theta x^\perp$. Then, the condition $T^k_\theta x\in\Fix T$ holds if and only if $x^\perp\in\ker T^k_\theta=\ker T_\theta$, which is equivalent to saying that $x^\perp$ belongs to the eigenspace of $T$ associated with the eigenvalue $\lambda:=1-1/\theta$. We distinguish two cases: (i) If $\theta\neq 1$, then $\lambda$ does not satisfy \cref{isoav lambda}, so $x^\perp = 0$ and thus, $x\in\Fix T$; (ii) If $\theta=1$, then $T_\theta=T$, and we deduce that $x\in\Fix T\oplus\ker T$.
\end{proof}

\begin{remark}\label{remark}
\begin{enumerate}[wide, labelwidth=0pt, labelindent=0pt]
      \setlength{\itemsep}{0pt}
  \setlength{\parskip}{0pt}
  \setlength{\parsep}{0pt}
    \item\label{remark1} Observe that~\cref{rate T theta} implies that for any relaxation parameter $\theta\in(0,2)\setminus\{1\}$, the relaxed map $T_\theta$ has a strictly slower rate of convergence than $T$.
    \item\label{remark2} For any $x\in\E$ and any $k\in\N$, one has that $\|T_\theta^{k}x-P_{\Fix T}x\|=\|T_\theta^{k}\overline{x}\|$, where $\overline{x}:=P_{(\Fix T)^\perp}x$. The convexity (\cref{convex}) and symmetry (\cref{symmetry}) of the function $\theta\mapsto\norm{T^k_\theta x}$ implies that the optimal relaxation parameter is $\theta=1$ (meaning that, among all parameters $\theta\in (0,2)$, the value $\theta=1$ provides the minimum distance of the iterates to the limit point $P_{\Fix T}x$). Moreover, for any $\theta\in(0,2)$, one has $\|T_\theta^{k}x-P_{\Fix T}x\|=\|T_{2-\theta}^{k}x-P_{\Fix T}x\|$. Thus, although the iterates $T_\theta^{k}x$ and $T_{2-\theta}^{k}x$ generally differ, their distance to the limit point is exactly the same (an illustrative example is provided in the next section, see~\cref{fig:geom}).
    \item\label{remark3} Since the Douglas--Rachford map is iso-averaged~\cite[Proposition~3.5 (iv)]{cosineDR}, \cref{symmetry} extends~\cite[Theorem~3.10]{relaxDR}.
  \end{enumerate}
\end{remark}

In the following section, we discuss the relationship between the iso-averagedness of the maps defined by the family of graph-based splitting algorithms and the property that the graph and subgraph defining the algorithms coincide, thereby answering the first open question in~\cite{graphnumerical}.

\section{Graph-based splitting algorithms applied to subspaces}\label{sec:graph}

Let $G=(\mathcal{N},\Ei)$ be a directed graph defined by a set of \emph{nodes} $\mathcal{N}:=\{1,\ldots,n\}$ and a set of \emph{edges} $\Ei\subseteq\mathcal{N}\times\mathcal{N}$. We recall some definitions to be used in the sequel:
\begin{itemize}
\item An \emph{algorithmic graph} $G=(\mathcal{N},\Ei)$ is a connected directed graph such that $(i,j)\in\Ei\Rightarrow i<j$.

\item The \emph{adjacency matrix} $\operatorname{Adj}(G)$ is the $n\times n$ matrix such that $\operatorname{Adj}(G)_{i,j}=1$ if and only if $(i,j)\in\Ei$ and zero otherwise.

\item The \emph{degree} of a node $i$ is $d_i:=|\{j\in\mathcal{N}\mid (i,j)\in\Ei\text{ or }(j,i)\in\Ei\}|$, where $|\cdot|$ denotes the cardinality of a set, and the \emph{degree matrix} is the diagonal $n\times n$ matrix $\operatorname{Deg}(G):=\Diag(d_1,\ldots,d_n)$. We emphasize that we use the definition for the underlying undirected graph of $G$.
\item The \emph{Laplacian matrix} of $G$ is defined as $\operatorname{Lap}(G):=\operatorname{Deg}(G)-\operatorname{Adj}(G)-\operatorname{Adj}(G)^T$. Again, we use the definition for the underlying undirected graph.
\end{itemize}
Throughout, we assume that $G=(\mathcal{N},\Ei)$ is an algorithmic graph and $G'\subseteq G$ is a connected spanning subgraph (i.e., $G'$ uses all nodes). Recall that $\operatorname{Lap}(G')$ is an $n\times n$ matrix with rank $n-1$ and $\ker(\Lap(G'))=\spa\{\mathbf{1}_n\}$, where $\mathbf{1}_n=(1,\ldots,1)\in\R^n$, and there exists a full-rank $n\times (n-1)$ matrix $Z$ such that
\begin{equation}\label{eq:decomp}
\operatorname{Lap}(G')=ZZ^T,
\end{equation}
see~\cite[Proposition~2.16]{graph-fb} for details.

Let $U_1,\ldots,U_n$ be linear subspaces of $\E$. The family of graph-based splitting methods introduced in~\cite{graph-drs}, when applied to the normal cones of these subspaces, can be described by the map $\mathcal{T}:\E^{n-1}\to\E^{n-1}$, given for $\bm{v}\in\E^{n-1}$ by
\begin{equation}\label{variable x}
    \mathcal{T}(\bm{v})=\bm{v}-\ot{Z^T}\bm{x},\text{ where }x_i:=P_{U_i}\left(\frac{2}{d_i}\sum_{(h,i)\in\Ei}x_h+\frac{1}{d_i}(\ot{Z}\bm{v})_i\right),\;i=1,\ldots, n,
\end{equation}
see also \cite[Proposition~3.5]{graph-linear}, where we emphasize that $d_i=\operatorname{Deg}(G)_{i,i}$ and $Z$ satisfies~\cref{eq:decomp}. Here $\ot{Z}:=Z\otimes\Id:\E^{n-1}\to\E^{n}$, which is defined for any $\bm{v}\in\E^{n-1}$ as
\begin{equation*}
  \ot{Z}\bm{v}:=\left(\sum_{j=1}^{n-1}Z_{1,j}v_j,\ldots,\sum_{j=1}^{n-1}Z_{n,j}v_j\right),
\end{equation*}
so the equality ${(\ot{Z})}^* = \ot{Z^T}$ holds.  Also, if we denote by $\Delta:=\{(x,\ldots,x)\mid x\in\E\}\subseteq\E^n$ the \emph{diagonal subspace}, by \cite[Proposition 2.18]{graph-linear}, the range of $\ot{Z}$ is
\begin{equation}\label{diagonal}
  \ran(\ot{Z})=\Delta^\perp.
\end{equation}

To fully describe $\mathcal{T}$ as a composition of maps, we first represent $\bm{x}$ as the image of $\bm{v}$ under a suitable map. By \cite[Equation~(3.11)]{graph-drs} the variable $\bm{x}$ from \cref{variable x}
can be rearranged as follows. Let $\mathcal{U}:=U_1\times\cdots\times U_n$ and $\oDelta:=\operatorname{Deg}(G)-2\operatorname{Adj}(G)^T$. Then $\bm{x}\in\mathcal{U}$ and
\begin{equation}\label{Zv=Dx}
	\quad \ot{Z}v-\ot{\oDelta}\bm{x}\in\mathcal{U}^\perp,
\end{equation}
which is equivalent to saying that
\begin{equation}\label{PUZDelta}
  P_\mathcal{U}(\ot{Z}\bm{v})=P_\mathcal{U}(\ot{\oDelta}\bm{x}).
\end{equation}
If we denote by $\mathcal{M}_\oDelta:=P_\mathcal{U}\ot{\oDelta} P_\mathcal{U}+P_{\mathcal{U}^\perp}:\E^n\to\E^n$, then \cref{PUZDelta} is equivalent to
\begin{equation}\label{eq:PUZv}
P_\mathcal{U}(\ot{Z}\bm{v})=\mathcal{M}_\oDelta\bm{x},
\end{equation}
since $\bm{x}\in\mathcal{U}$. The following lemma proves that, among other properties, $\mathcal{M}_\oDelta$ is invertible on $\mathcal{U}$, which allows us to solve \cref{eq:PUZv} for the variable $\bm{x}$.

\begin{lemma}\label{M}
The linear map $\mathcal{M}_\oDelta:\E^n\to\E^n$ is invertible, and the following assertions hold:
\begin{enumerate}
	\item ${(\mathcal{M}_\oDelta)}^*=\mathcal{M}_{\oDelta^T}$.\label{M1}
	\item $\mathcal{M}_\oDelta^{-1}$ and $P_\mathcal{U}$ commute.\label{M2}
	\item $\mathcal{M}_\oDelta^{-1}P_\mathcal{U}\ot{\oDelta}P_\mathcal{U}=P_\mathcal{U}\ot{\oDelta}P_\mathcal{U}\mathcal{M}_\oDelta^{-1}=P_\mathcal{U}$.\label{M3}
\end{enumerate}
\end{lemma}

\begin{proof}
  Notice that $\mathcal{M}_\oDelta$ acts as the identity on $\mathcal{U}^\perp$. It suffices to show that the restriction of $\mathcal{M}_\oDelta$ to $\mathcal{U}$, given by $\mathcal{M}_\oDelta|_\mathcal{U}=P_\mathcal{U}\ot{\oDelta}$, is invertible on $\mathcal{U}$. Let $\bm{u}\in\mathcal{U}$ and $\widetilde{\bm{u}}:=P_\mathcal{U}\ot{\oDelta}\bm{u}$. Then, the components of $\widetilde{\bm{u}}$ are described as
  \begin{equation}\label{PUi}
    \widetilde{u}_i=P_{U_i}\bigg(d_iu_i-2\sum_{(h,i)\in\Ei}u_h\bigg)=d_iu_i-2\sum_{(h,i)\in\Ei}P_{U_i}(u_h).
  \end{equation}
  Since $d_i>0$ and $(h,i)\in\Ei$ implies $h<i$, we conclude by forward substitution that each component of $\bm{u}$ can be written as a linear combination of the components of $\widetilde{\bm{u}}$.

  We now verify the listed properties. First, \cref{M1} holds by the definition of $\mathcal{M}_\oDelta$ and the self-adjointness of the orthogonal projection. Next, to verify \cref{M2}, we observe that this condition is equivalent to $\mathcal{M}_\oDelta P_\mathcal{U}=P_\mathcal{U}\mathcal{M}_\oDelta$, which clearly holds. Finally, \cref{M3} follows from the identity $P_\mathcal{U}\ot{\oDelta} P_\mathcal{U} =\mathcal{M}_\oDelta P_\mathcal{U}$.
\end{proof}

Using  \cref{eq:PUZv} and \cref{M}, we can finally express the graph operator $\mathcal{T}$ given in \cref{variable x} as a composition of linear maps. Specifically,
\begin{equation}\label{eq:T}
\mathcal{T}=\Id-\oxi,\text{ with }\oxi:=\ot{Z^T}\mathcal{M}_\oDelta^{-1} P_\mathcal{U}\ot{Z},
\end{equation}
where $\Id$ denotes now the identity map in $\E^{n-1}$.

\begin{theorem}\label{G=G'}
    The linear map $\mathcal{T}:\E^{n-1}\to\E^{n-1}$ in~\cref{variable x} is iso-averaged for all linear subspaces $U_1,\ldots,U_n$ if and only if $G=G'$.
\end{theorem}

\begin{proof}
  By \cref{eq:T}, it is easy to show that the map $\mathcal{T}$ is iso-averaged if and only if $\oxi$ is iso-averaged. Observe that, by \crefpart{M}{M1}, we have $\oxi^*=\ot{Z^T}P_\mathcal{U}\mathcal{M}^{-1}_{\oDelta^T}\ot{Z}$. Recalling~\cref{eq:decomp}, we obtain
  \begin{equation}\label{xi*xi}
    \oxi^*\oxi=\ot{Z^T}\mathcal{M}_{\oDelta^T}^{-1} P_\mathcal{U}\ot{\Lap(G')}P_\mathcal{U}\mathcal{M}^{-1}_\oDelta \ot{Z}.
  \end{equation}
  Given any matrix $K\in\R^{n\times n}$, and denoting by
  \begin{equation}\label{chi}
    \ochi_K:=\mathcal{M}_{\oDelta^T}^{-1} P_\mathcal{U}\ot{K}P_\mathcal{U}\mathcal{M}^{-1}_\oDelta={\left(P_\mathcal{U}\mathcal{M}_{\oDelta}^{-1}\right)}^*\ot{K}\left(P_\mathcal{U}\mathcal{M}^{-1}_\oDelta\right),
  \end{equation}
  we simplify \cref{xi*xi} as $\oxi^*\oxi=\ot{Z^T}\ochi_{\Lap(G')}\ot{Z}$.

  Now, recall that $\Lap(G')=\Lap(G)-\Lap(G\setminus G')$, where $G\setminus G'$ is the graph $G$ without the edges of $G'$. Also, it holds that $\Lap(G)=(\oDelta+\oDelta^T)/2$ (see \cite[Remark~3.12]{graph-fb}), so we get $2\ochi_{\Lap(G)}=\ochi_\oDelta+\ochi_{\oDelta^T}$. Combining all of this with \cref{xi*xi}, we obtain
  \begin{equation}
    2\oxi^*\oxi=\ot{Z^T}\big(\ochi_\oDelta+\ochi_{\oDelta^T}-2\ochi_{\Lap(G\setminus G')}\big)\ot{Z}.
  \end{equation}
  By \crefpart{M}{M2} and \cref{M3}, we write $\ochi_{\oDelta}=\mathcal{M}_{\oDelta^T}^{-1} (P_\mathcal{U}\ot{\oDelta} P_\mathcal{U}\mathcal{M}^{-1}_\oDelta)=\mathcal{M}_{\oDelta^T}^{-1}P_\mathcal{U}=P_\mathcal{U}\mathcal{M}_{\oDelta^T}^{-1}$. Notice in this case that $\ot{Z^T}\ochi_\oDelta \ot{Z}=\oxi^*$. Similarly, we get $\ochi_{\oDelta^T}=\mathcal{M}_{\oDelta}^{-1}P_\mathcal{U}$ and $\ot{Z^T}\ochi_{\oDelta^T} \ot{Z}=\oxi$. Combining everything, we obtain
\begin{equation}
  2\oxi^*\oxi=\oxi^*+\oxi-2\ot{Z^T}\ochi_{\Lap(G\setminus G')}\ot{Z}.
\end{equation}
Therefore, $\oxi$ is iso-averaged if and only if
\begin{equation}\label{eq:equiv_iso}
\ot{Z^T}\ochi_{\Lap(G\setminus G')}\ot{Z}=0.
\end{equation}

If $G=G'$, then~\cref{eq:equiv_iso} trivially holds. Assume now that $\oxi$ is iso-averaged for all subspaces $\mathcal{U}$. Since $\ot{\Lap(G\setminus G')}$ is positive semidefinite, by \cref{chi} and \cite[Observation 7.1.8]{matrixAnalysisHornJohnson}, we have that $\ot{Z^T}\ochi_{\Lap(G\setminus G')}\ot{Z}$ is also positive semi-definite and
\begin{align}\label{eq:kernels}
 \begin{aligned}
  \ker \left(\ot{Z^T}\ochi_{\Lap(G\setminus G')}\ot{Z}\right)
  &=\ker\left({\left(P_\mathcal{U}\mathcal{M}_B^{-1}\ot Z\right)}^*\ot{\Lap(G\setminus G')}\left(P_\mathcal{U}\mathcal{M}^{-1}_B\ot{Z}\right)\right)\\
  &=\ker\left(\Lap(G\setminus G')_{\scaleto{\otimes}{4pt}}\mathcal{M}^{-1}_{\oDelta}P_\mathcal{U}\ot{Z}\right),
 \end{aligned}
\end{align}
where \crefpart{M}{M2} has also been used in the last equality. Then, \cref{eq:kernels} implies that~\cref{eq:equiv_iso} is equivalent to $\ot{\Lap(G\setminus G')}\mathcal{M}^{-1}_{\oDelta}P_\mathcal{U}\ot{Z}=0$. Consequently, \cref{diagonal} yields the equivalent inclusion
\begin{equation}\label{inclusion}
  P_\mathcal{U}(\Delta^\perp)=P_\mathcal{U}(\ran \ot{Z})=\ran(P_\mathcal{U}\ot{Z})\subseteq\ker\big(\ot{\Lap(G\setminus G')}\mathcal{M}^{-1}_{\oDelta}\big)=\mathcal{M}_B\big(\okernel\big),
\end{equation}
where $\okernel:=\ker\big(\ot{\Lap(G\setminus G')}\big)$. By definition of $\mathcal{M}_\oDelta$, we get $\mathcal{M}_\oDelta\big(\okernel\big)=P_\mathcal{U}\ot{\oDelta}P_\mathcal{U}(\okernel)\oplus(\mathcal{U}^\perp\cap\okernel)$. Thus, \cref{inclusion} can be rewritten as
\begin{equation}\label{inclusion2}
  P_\mathcal{U}(\Delta^\perp)\subseteq P_\mathcal{U}\ot{\oDelta}P_\mathcal{U}(\okernel)\Leftrightarrow \Delta^\perp \subseteq \ot{\oDelta}P_\mathcal{U}(\okernel)+\mathcal{U}^\perp\Leftrightarrow(\ot{\oDelta}P_\mathcal{U}(\okernel))^\perp\cap\mathcal{U}\subseteq\Delta.
\end{equation}
For each $i=1,\ldots,n$, let us choose in \cref{inclusion2} the subspace $\mathcal{U}^i\subseteq \E^n$ as the Cartesian product of all trivial subspaces except for the $i$-th component, which is equal to $\E$. Then $\Lambda_i:=P_{\mathcal{U}^i}(\Lambda)$ is the projection of $\Lambda$ onto the $i$-th coordinate. Thus, $(\ot{\oDelta}(\Lambda_i))^\perp\cap\mathcal{U}_i$ is the Cartesian product of all trivial subspaces except for the $i$-th component, which is $(B_{i,i}\Lambda_i)^\perp=(d_i\Lambda_i)^\perp=\Lambda_i^\perp$. Since this subspace is contained in $\Delta$, this means that $\Lambda_i^\perp=\{0\}$ for all $i=1,\ldots,n$. Therefore, $\Lambda^\perp=\{0\}^n$ and $\okernel=\ker\big(\ot{\Lap(G\setminus G')}\big) =\E^n$, which implies that $G=G'$.
\end{proof}

\begin{remark}
The equivalent inclusions in~\cref{inclusion2} show that $Z$ does not really play a role in the iso-averagedness of $\mathcal{T}$. Further, if $\widetilde{Z} = Z O$ (with $O$ orthogonal), then $\widetilde{\mathcal{C}}=\ot{\widetilde{Z}^T}\mathcal{M}_\oDelta^{-1} P_\mathcal{U}\ot{\widetilde{Z}}= \ot{O^T} \mathcal{C} \ot{O}$ and, therefore, the normality of $\mathcal{T}$ is also independent of the specific decomposition $Z$ of $\operatorname{Lap}(G')$.
\end{remark}

Thanks to \cref{G=G'}, we are now able to answer the first open question in \cite{graphnumerical} and also explain the symmetry observed in \cite[Figure 2]{graphnumerical}. Indeed, if $G=G'$, then $\mathcal{T}$ is iso-averaged and, together with \crefpart{remark}{remark2}, we deduce the following result.

\begin{corollary}\label{cor:main}
  If $G=G'$, then the optimal relaxation parameter for the graph-based splitting algorithms is $\theta=1$. Furthermore, the distance to the limit point of the iterates generated by the algorithms with parameters $\theta$ and $2-\theta$ coincide.
\end{corollary}

The next example provides a geometrical illustration of our findings.

\begin{example}\label{ex:geometric} Let $G=G'$ be the sequential graph (in which each node is connected to the subsequent one) with 3 nodes. This gives rise to the linear map $\mathcal{T}:\E^2\to\E^2$ given by
\begin{equation*}
  \mathcal{T}(\bm{v})=\bm{v}-\begin{pmatrix} x_1-x_2\\ x_2-x_3\end{pmatrix},\text{ where }\left\{\begin{array}{l}
  x_1=P_{U_1}(v_1),\\
  x_2=P_{U_2}\left(x_1+\frac{1}{2}(v_2-v_1)\right),\\
  x_3=P_{U_3}(2x_2-v_2).
  \end{array}\right.
\end{equation*}
By~\cite[Proposition~4.3 and Section~4.5]{graph-linear}, one has $\Fix\mathcal{T}=\ot{(\mathbf{1}_2)}(U_1\cap U_2\cap U_3)\oplus E$, where
\begin{equation*}
  E=\left\{\bm{e}\in \E^2\;\bigl\lvert\; e_1\in U_1^\perp, e_2-e_1\in U_2^\perp, -e_2\in U_3^\perp\right\},
\end{equation*}
and $\oplus$ denotes the Minkowski sum with orthogonal summands.

Consider the problem of finding a point in the intersection of three lines in $\E=\R^2$ passing through the origin. Specifically,  let $U_i:=\{a_i\}^\perp$ for some vectors $a_i\in\R^2$, $i=1,2,3$, with $a_1$ and $a_2$ linearly independent. After some algebraic manipulations, we obtain
\begin{equation*}
  E=\spa\left\{\begin{pmatrix} \mu a_1\\ a_3\end{pmatrix}\right\},\text{ where }\mu:=\frac{\det[a_3,a_2]}{\det[a_1,a_2]}.
\end{equation*}
Since $U_1\cap U_2\cap U_3=\{0\}$, we deduce that $\Fix\mathcal{T}=E$. By~\cite[Theorem~4.8]{graph-linear}, for any starting point $\bm{v}^0=(v^0_1,v^0_2)\in\R^2\times\R^2$, the governing sequence $\bm{v}^k=\mathcal{T}^k(\bm{v}^0)$ converges as $k\to\infty$ to
\begin{equation*}
  \bm{e}^\star:=P_E(\bm{v}^0)=\frac{\mu\langle a_1,v_1^0\rangle +\langle a_3,v_2^0\rangle}{\mu^2\|a_1\|^2+\|a_3\|^2}\begin{pmatrix} \mu a_1\\ a_3\end{pmatrix}.
\end{equation*}

For illustration, for a given starting point $\bm{v}^0\in\R^2\times\R^2$ and three specific lines, we computed the iterates generated for three relaxation parameters, namely, $0.2$, $1$, and $2-0.2$. Since $\theta=1$ is the optimal relaxation parameter, it minimizes the function 
$$f_k(\theta):=\|\mathcal{T}_\theta^k (\bm{v}^0-\bm{e}^\star)\|=\|\mathcal{T}_\theta^k (\bm{v}^0)-\bm{e}^\star\|,\quad\text{for all }k\in\N.$$ In particular, for $k=1$, it minimizes $(f_1(\theta))^2=\|\mathcal{T}_\theta (\bm{v}^0)_1-\bm{e}_1^\star\|^2+\|\mathcal{T}_\theta (\bm{v}^0)_2-\bm{e}_2^\star\|^2$, whose value is equal to the sum of the areas of the two squares determined by $\{\mathcal{T}_\theta (\bm{v}^0)_1,\bm{e}_1^\star\}$ and $\{\mathcal{T}_\theta (\bm{v}^0)_2,\bm{e}_2^\star\}$ (see \cref{fig:geom_a}). Further, \cref{cor:main} shows that $f_1(\theta)=f_1(2-\theta)$, so the sum of the areas of the two respective pairs of squares (blue and green) coincides. On the other hand, in \cref{fig:geom_b}, we represent the points $p_\theta^k:=(\|\mathcal{T}^k_\theta (\bm{v}^0)_1-\bm{e}_1^\star\|,\|\mathcal{T}^k_\theta (\bm{v}^0)_2-\bm{e}_2^\star\|)$ for $k=1,\ldots,10$. The property $f_k(\theta)=f_k(2-\theta)$ is depicted by the fact that $p_{\theta}^k$ and $p_{2-\theta}^k$ have the same norm for $\theta=0.2$, i.e., they lie on the same circle of center the origin and radius $f_k(\theta)$. Lastly, we observe that setting the relaxation parameter to $1$ provides a clear advantage, as the algorithm converges to the limit point~$\bm{e}^\star$ faster than with the other two choices.

\begin{figure}\centering

  \begin{subfigure}{0.49\textwidth}
    \centering
    \includegraphics[height=6.8cm]{Sequential_e_star.pdf}
    \caption{The value $\|\mathcal{T}_\theta^k(\bm{v})-\bm{e}^\star\|^2$ can be interpreted as the sum of the areas of two squares, which is minimized in the unrelaxed case (red squares)}\label{fig:geom_a}
  \end{subfigure}\hfill
  \begin{subfigure}{0.49\textwidth}
    \centering
    \includegraphics[height=6.8cm]{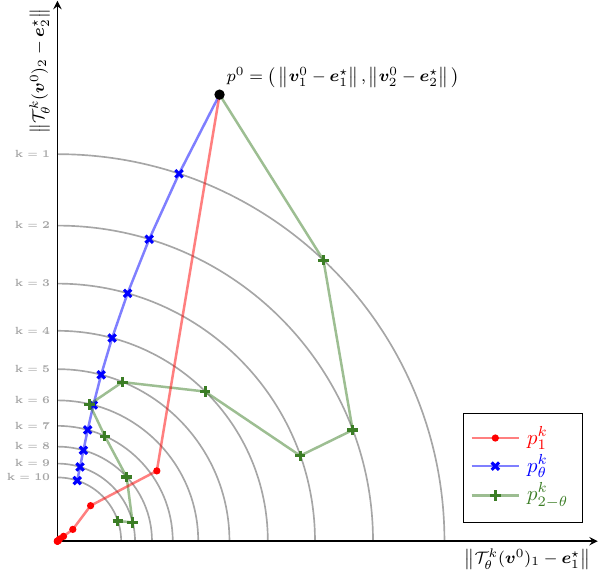}
    \caption{The distance of the iterates to the limit point coincides for $\theta$ and $2-\theta$, while the unrelaxed map provides the best performance\label{fig:geom_b}}
  \end{subfigure}
\caption{Geometrical interpretation of \cref{ex:geometric}, where $a_1=(-1,6)$, $a_2=(-3,1)$, $a_3=(-3,-4)$, $\bm{v}^0=\big((1,-10),(-8,1)\big)$, and $\theta=0.2$\label{fig:geom}}
\end{figure}
\end{example}

We conclude by presenting some examples of pairs of graphs $(G,G')$ for which the linear map $\mathcal{T}$ fails to be normal or iso-averaged, as well as an example in which it is iso-averaged but $G\neq G'$.

\begin{example}[Lack of normality and iso-averagedness of $\mathcal{T}$]
Let $\mathcal{U}=\E^n$, with $n\geq 4$. Then $\oxi=\ot{C}$, with
    \[C = Z^T\oDelta^{-1} Z,\]
and thus, $\mathcal{T}$ is normal or iso-averaged if and only if $C$ is so.
\begin{itemize}[leftmargin=*]
    \item \emph{$\mathcal{T}$ is not normal.} Let $G'=(\mathcal{N},\Ei')$ be the parallel down graph (a star graph whose center is the last node) and let $G=(\mathcal{N},\Ei)$ with $\Ei=\Ei'\cup\{(1,2)\}$. Then $B$ and $Z$ are given by blocks as follows
        $$B=\begin{bmatrix*}
        2&0 &\mathbf{0}_{n-3}^T&0\\
        -2& 2 & \mathbf{0}_{n-3}^T&0\\
        \mathbf{0}_{n-3} & \mathbf{0}_{n-3} & \Id_{n-3}&\mathbf{0}_{n-3}\\
        -2 & -2 & -2\cdot \mathbf{1}_{n-3}^T & n-1
        \end{bmatrix*}\text{ and }Z=\begin{bmatrix*}\Id_{n-1}\\ -\mathbf{1}_{n-1}^T\end{bmatrix*},$$
        from where
        $$C=\frac{1}{n-1}\begin{bmatrix*}
        \frac{n-3}{2} & 0 & -\mathbf{1}_{n-3}^T\\ \frac{n-3}{2} & \frac{n-1}{2} & -\mathbf{1}_{n-3}^T\\ -\mathbf{1}_{n-3} & \mathbf{0}_{n-3} & (n-1)\Id_{n-3}-\mathbf{1}_{(n-3)\times(n-3)}
        \end{bmatrix*}.$$
        Since the squared norm of the first row of $C$ differs from that of the first column, we have $C^TC\neq CC^T$, and thus $C$ is not normal.
    \item \emph{$\mathcal{T}$ is normal but not iso-averaged.} Let $G'$ be the parallel up graph (a star graph whose center is the first node) and let $G$ be the biparallel graph (the union of $G'$ and the parallel down).
    Then $B$ and $Z$ can be described by blocks as follows:
    \[B=\begin{bmatrix}
        n-1 & \mathbf{0}_{n-2}^T & 0 \\
        -2\cdot\mathbf{1}_{n-2} & 2\Id_{n-2} & \mathbf{0}_{n-2} \\
        -2 & -2\cdot\mathbf{1}_{n-2}^T & n-1
    \end{bmatrix}
        \text{ and }
        Z=\begin{bmatrix}
        \mathbf{1}_{n-1}^T \\
        -\Id_{n-1}
        \end{bmatrix}\Rightarrow C=\begin{bmatrix}
          \frac12 \Id_{n-2} & \mathbf{0}_{n-2}\\ \mathbf{0}_{n-2}^T & 0\end{bmatrix}.
          \]
          Since $C$ is symmetric, then it is normal. However, we can observe that $C$ is not iso-averaged, since the eigenvalue $\lambda=\frac12$ does not satisfy~\cref{isoav lambda}.
    \item \emph{$\mathcal{T}$ is iso-averaged with $G\neq G'$.} Let $G'=(\mathcal{N},\Ei')$ be the sequential and let $G=(\mathcal{N},\Ei)$ be the ring graph, with $\Ei=\Ei'\cup \{(1,n)\}$. Notice that the resultant algorithm is known as the Malitsky--Tam algorithm (see \cite[Algorithm 1]{malitsky2023resolvent}). In this case, $Z$ is given by $Z_{j,j}=1$, $Z_{j+1,j}=-1$, for $j=1,\ldots,n-1$, and zero otherwise. On the other hand, $B$ is defined componentwise as $B_{i,i}=2$  for $i=1,\ldots,n$, $B_{i+1,i}=B_{n,1}=-2$ for $i=1,\ldots,n-1$, and zero otherwise. Hence, one computes
    \[C_{i,j}=\begin{cases}
      \frac12 & \text{if }i=j,\\
      -\frac12 & \text{if }i=j-1\text{ or }(i,j)=(n-1,1),\\
      0&\text{otherwise}.
    \end{cases}\]
    It is not difficult to check that $C$ is iso-averaged, since $C=\frac{1}{2}\bigl(\Id+(-P)\bigr)$, where $P$ is a permutation matrix (and thus orthogonal).
\end{itemize}
\end{example}

\paragraph{Acknowledgements} The authors were partially supported by Grant PID2022-136399NB-C21 and Grant PID2025-168246NB-I0 funded by MICIU/AEI/10.13039/501100011033 and by ERDF/EU. César López-Pastor was supported by Grant PREP2022-000118 funded by ``ESF Investing in your future'' and by MICIU/AEI/10.13039/501100011033. The authors are also very grateful to the anonymous referees for their helpful remarks and suggestions, which allowed us to improve the original manuscript.

\printbibliography

\end{document}